\documentclass[letterpaper, 11 pt]{article}
\usepackage{fullpage,amsthm, amsmath, amssymb, amsfonts}
\usepackage{graphicx}

\newcommand{\RR}{\mathbb{R}}
\newcommand{\NN}{\mathbb{N}}

\newcommand{\Ass}{Ass}
\newcommand{\CvxH}{\text{ConvexHull}}

\newcommand{\C}{\mathcal{C}}

\newcommand{\I}{\mathcal{I}}

\newcommand{\A}{\mathcal{A}}

\newcommand{\newword}[1]{\textbf{\emph{#1}}}

\newtheorem{conj}{Conjecture}[section]

\newtheorem{theorem}[conj]{Theorem}

\newtheorem{proposition}[conj]{Proposition}

\newtheorem{lemma}[conj]{Lemma}

\newtheorem{corollary}[conj]{Corollary}

\newtheorem{problem}[conj]{Problem}
\newtheorem{remark}[conj]{Remark}

\begin{document}

\title{Poset vectors and generalized permutohedra}
\author{Dorian Croitoru, SuHo Oh and Alexander Postnikov}

\date{}
\maketitle
\abstract{We show that given a poset $P$ and and a subposet $Q$, the integer points obtained by restricting linear extensions of $P$ to $Q$ can be explained via integer lattice points of a generalized permutohedron.}


\section{Introduction}

Let $\lambda = (\lambda_1,\ldots,\lambda_n)$ be a partition with at most $n$ parts. The \newword{Young diagram} of shape $\lambda$ is the set
$$D_{\lambda} = \{(i,j) \in \NN^2 | 1 \leq j \leq n, 1 \leq i \leq \lambda_j\}.$$
A \newword{Standard Young tableau} is a bijective map $T : D_{\lambda} \rightarrow \{1,\ldots,|D_{\lambda}|\}$ which is increasing along rows and down columns, i.e. $T(i,j) < T(i,j+1)$ and $T(i,j) < T(i+1,j)$ \cite{stanley2000enumerative}. Standard Young tableaus of $\lambda = (2,2,1)$ are given in Figure~\ref{fig:young}. In each tableau, the entries at boxes $(1,2)$ and $(2,2)$ are colored with red. The pairs we get from each tableau, are exactly the integer lattice points of a pentagon in the right image of Figure~\ref{fig:young}. Then one could naturally ask the following question : If we choose some arbitrary boxes inside a Young diagram, and collect the integer vectors we get from the chosen boxes for each standard Young diagram, are they the integer lattice points of some polytope?

Such questions were studied for diagonal boxes of shifted Young diagrams by the first author and the third author in  \cite{Postnikov01012009} and \cite{2008arXiv0803.2253C}. Let $\lambda = (\lambda_1,\ldots,\lambda_n)$ be a partition with at most $n$ parts. The \newword{shifted Young diagram} of shape $\lambda$ is the set
$$SD_{\lambda} = \{(i,j) \in \NN^2 | 1 \leq j \leq n, j \leq i \leq n + \lambda_j \}.$$

\begin{figure}[htbp]
	\begin{center}
	 		\includegraphics[width=0.9\textwidth]{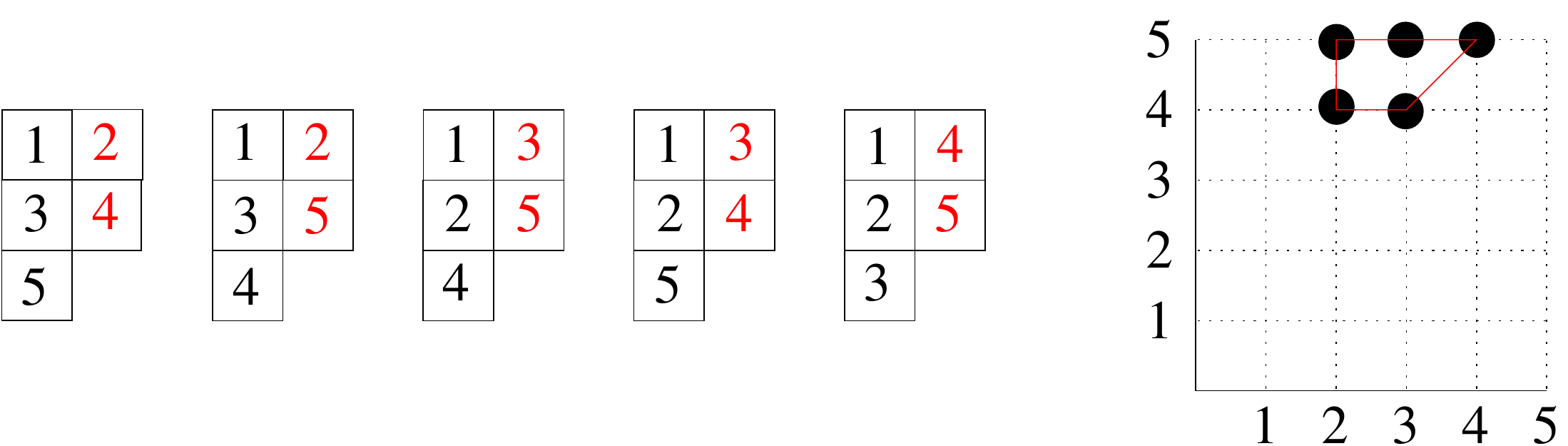}
		\caption{Example of standard Young tableaus of $\lambda = (2,2,1)$, and pairs of entries that occur at $(1,2)$ and $(2,2)$ inside the tableaus.}
		\label{fig:young}
	\end{center}
\end{figure}

We think of $SD_{\lambda}$ as a collection of boxes with $n+1-i-\lambda_i$ boxes in row $i$, such that the leftmost box of the $i$-th row is also in the $i$-th column. A \newword{shifted standard Young tableau} is a bijective map $T : SD_{\lambda} \rightarrow \{1,\ldots,|SD_{\lambda}|\}$ which is increasing along rows and down columns, i.e. $T(i,j) < T(i,j+1)$ and $T(i,j) < T(i+1,j)$. The \newword{diagonal vector} of such a tableau $T$ is $Diag(T) = (T(1,1),T(2,2),\ldots,T(n,n))$.

\begin{figure}[htbp]
	\begin{center}
	 		\includegraphics[width=0.3\textwidth]{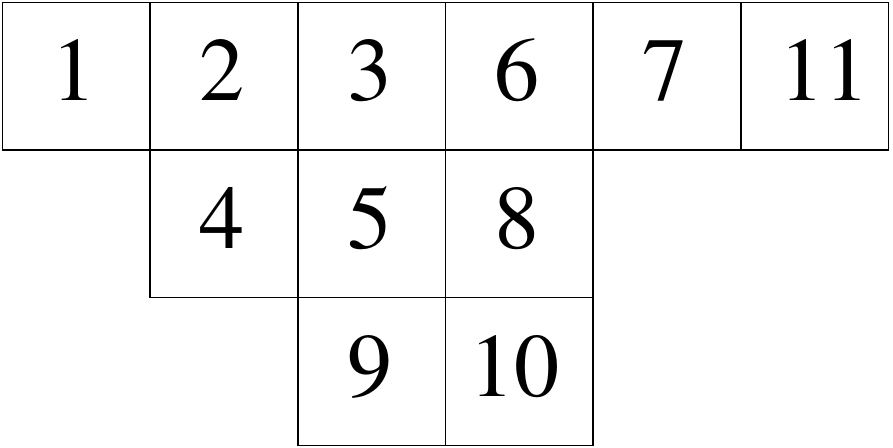}
		\caption{Example of a shifted standard Young tableau, which has diagonal vector $(1,4,9)$.}
		\label{fig:dorianex1}
	\end{center}
\end{figure}


Figure~\ref{fig:dorianex1} is a shifted standard Young tableau for $n=3, \lambda = (3,1,1)$. In \cite{Postnikov01012009}, the third author showed that when $\lambda = (0,\ldots,0)$, the diagonal vectors of $SD_{\emptyset}$ are in bijection with lattice points of $(n-1)$-dimensional \newword{associahedron} $\Ass_{n-1}$. Extending this result, the first author, in \cite{2008arXiv0803.2253C}, showed that the diagonal vectors of $SD_{\lambda}$ in general, are in bijection with lattice points of a certain deformation of the associahedron.

In this paper, we generalize the previous question for Young diagrams and the previous results for shifted Young diagrams, by looking at an arbitrary poset $P$ in general. More precisely, given an arbitrary poset $P$, a \newword{linear extension} is an order preserving bijection $\sigma : P \rightarrow [|P|]$, where $[n]$ is defined to be the set of integers $\{1,\ldots,n\}$. Let $Q$ be a subposet of $P$ and label the elements of $Q$ by $q_1,\ldots,q_{|Q|}$, such that if $q_i < q_j$ in $Q$, then $i<j$. We call a vector $(\sigma(q_1),\sigma(q_2),\ldots,\sigma(q_{|Q|}))$ obtained in such manner as the $(P,Q)$-\newword{subposet vector}.

\begin{figure}[htbp]
	\begin{center}
	 		\includegraphics[width=0.3\textwidth]{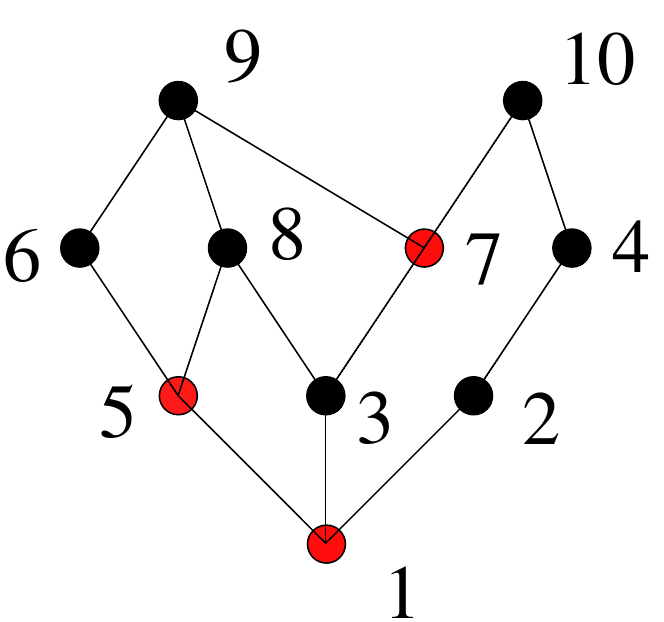}
		\caption{Example of poset $P$ and a subposet $Q$, where elements of $Q$ are colored red.}
		\label{fig:posetex1}
	\end{center}
\end{figure}

Figure~\ref{fig:posetex1} is a poset $P$ with $10$ elements. The elements of $Q$ are colored red. We label the elements of $Q$ by $q_1,q_2,q_3$ by starting from the lowest element going up. Then the $(P,Q)$-subposet vector we get in this case is $(1,5,7)$.

When we are only dealing with the linear extensions of $P$ (when $P=Q$), the connection between linear extensions of posets and generalized permutohedra has been studied in \cite{Postnikov2008} and \cite{Morton:2009:CRT:1718590.1718591}. In particular, when $Q$ is a chain, we will show that the $(P,Q)$-subposet vectors are in bijection with a certain deformation of the associahedron (generalized permutohedron). It has been shown in \cite{Stanley198156} that the number of linear extensions corresponding to a fixed $(P,Q)$-subposet vector are log-concave.

In the general case, we will show that the set of $(P,Q)$-subposet vectors can be thought as lattice points of a non-convex polytope, obtained by gluing the generalized permutohedra. In section 2, we will go over the basics of generalized permutohedra. In section 3, we will study the case when $Q$ is a chain. In section 4, we will go over the general case when $Q$ is a general subposet of $P$. In section 5, we give a nice combinatorial method to describe the vertices of the constructed polytope. 


\section{Generalized Permutohedron}

In this section, we will give an introduction to the \newword{associahedron} using \newword{generalized permutohedra} language from \cite{Postnikov01012009}. 

Associahedron, also known as the \newword{Stasheff polytope}, first appeared in the work of \cite{stasheff}. Given $n$ letters, think of inserting opening and closing parentheses so that each opening parenthesis is matched to a closing parentheses. Then the associahedron is the convex polytope in which each vertex corresponds to a way of correctly inserting the parentheses, and the edges correspond to single application of the associativity rule. But since we will be working with the integer lattice points of certain realization of an associahedron, we are going to be using a different definition using \newword{generalized permutohedra}.

The \newword{permutohedron} is the polytope obtained by the convex hull of vertices which are formed by permuting the coordinates of the vector $(1,2,\ldots,n)$. Hence the vertices correspond to each permutation of $S_n$, and the edges correspond to applying a transposition. The \newword{generalized permutohedra}, which was introduced in \cite{Postnikov01012009}, are polytopes that can be obtained by moving vertices of the usual permutohedron so that directions of all edges are preserved. It happens that for a certain class of generalized permutohedra, we can construct them using Minkowski sum of certain simplices.

Let $\Delta_{[n]} = \CvxH(e_1,\ldots,e_n)$ be the standard coordinate simplex in $\RR^n$. For a subset $I \subset [n]$, let $\Delta_I = \CvxH(e_i | i \in I)$ denote the face of $\Delta_{[n]}$. When $\I = (I_1,\ldots,I_t)$ where $I_i$'s are subsets of $[n]$, we denote $G_{\I}$ to be the Minkowski sum of $\Delta_{I_i}$'s. In other words, we have:
$$G_{\I} := \Delta_{I_1} + \cdots + \Delta_{I_t}.$$
Since the $I_i$'s do not have to be distinct, we could re-write the above sum as
$$G_{\I} := c_1\Delta_{I_1} + \cdots + c_m\Delta_{I_m},$$
where $c_i$ counts the number of times $I_i$ occurs among $\I$.


For convenience, unless otherwise stated, whenever we use the word generalized permutohedra, we will be referring to the class of polytopes that can be obtained via the construction above. Below are well known cases of generalized permutohedra. For more details, check Section 8 of \cite{Postnikov01012009}.


\newword{Permutohedron : } If we set $\I$ to consist of all possible nonempty subsets of $[n]$, then $G_{\I}$ is combinatorially equivalent to the usual permutohedron obtained by permuting the entries of point $(1,\ldots,n)$. 

\newword{Associahedron : } If we set $\I$ to consist of all possible intervals of $[n]$ (so that $[i,j]:=\{i,i+1,\ldots,j\}$ is in $\I$ for all pairs $i<j$), then $G_{\I}$ is combinatorially equivalent to the associahedron.


In this paper, we will mainly be dealing with generalized permutohedra, that can be obtained from the associahedron by deforming the facets.


\begin{figure}[htbp]
	\begin{center}
	 		\includegraphics[width=0.5\textwidth]{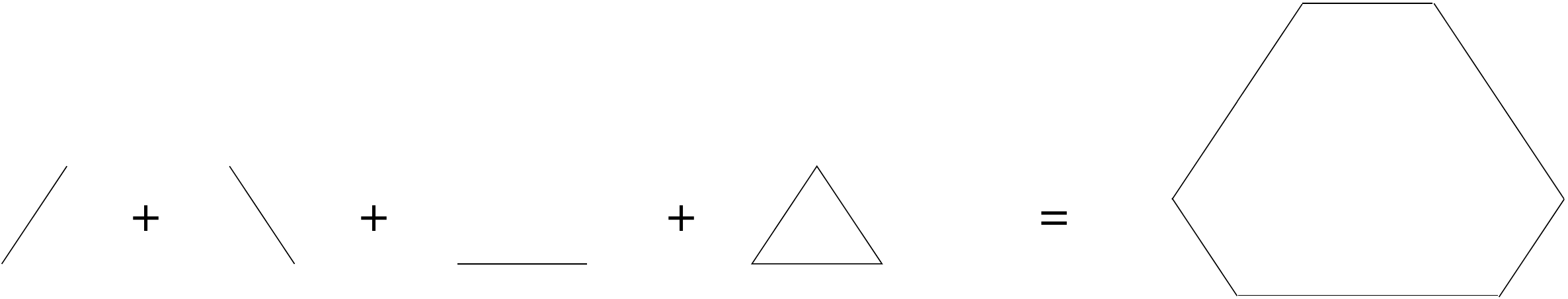}
		\caption{The permutohedron $G_{(\{1,2\},\{1,3\},\{2,3\},\{1,2,3\})}$.}
		\label{fig:ass3}
	\end{center}
\end{figure}

Figure~\ref{fig:ass3} shows an example of a permutohedron constructed by summing up all subsets of $[3]$. The terms $\Delta_{\{1\}},\Delta_{\{2\}},\Delta_{\{3\}}$ are omitted since summing points just corresponds to the translation of the polytope.

\begin{figure}[htbp]
	\begin{center}
	 		\includegraphics[width=0.5\textwidth]{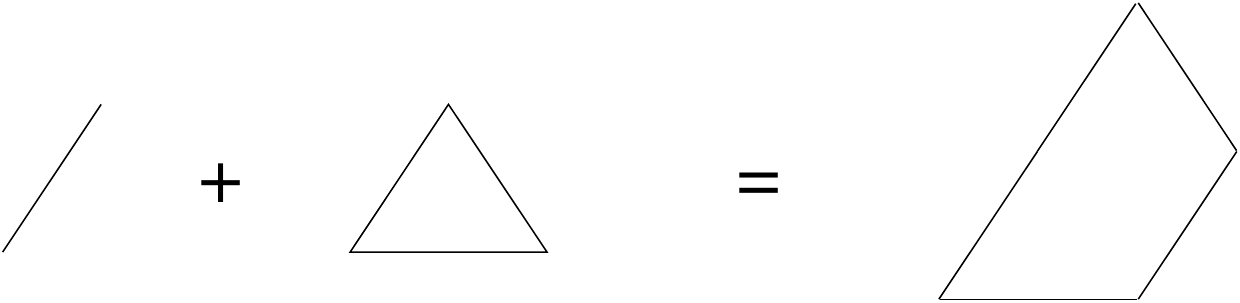}
		\caption{A deformed associahedron $G_{(\{1,2\},\{1,2,3\})}$.}
		\label{fig:gass3}
	\end{center}
\end{figure}

Figure~\ref{fig:ass3ver} shows an example of an associahedron constructed by summing $\Delta_{\{1,2\}}, \Delta_{\{2,3\}}$ and $\Delta_{\{1,2,3\}}$. Figure~\ref{fig:gass3} shows an example of a deformed associahedron constructed by summing $\Delta_{\{1,2\}}$ and $\Delta_{\{1,2,3\}}$. One can notice that the polytope in Figure~\ref{fig:gass3} can be obtained from the polytope in Figure~\ref{fig:ass3} or the polytope in Figure~\ref{fig:ass3ver} by moving around the facets.

\begin{lemma}[\cite{Postnikov01012009} Proposition 6.3]
\label{lem:facets}
Let $G_{\I}$ be a generalized permutohedron given by $c_1\Delta_{I_1} + \cdots + c_m\Delta_{I_m}$, where all $c_i$'s are positive integers. Then this polytope has the form $\{(t_1,\ldots,t_n) \in \RR^n | \sum t_i = z_{[n]}, \sum_{i \in I} t_i \leq z_I, \forall I\}$, where $z_I := \sum_{I_j \cap I \not = \emptyset} c_j$.
\end{lemma}

The above lemma allows us to obtain the defining hyperplanes of a generalized permutohedron. For example, if we look at $\Delta_{\{1,2\}} + \Delta_{\{1,2,3\}}$, then this polytope is the collection of points $(t_1,t_2,t_3)$ in $\RR^3$ such that:
\begin{itemize}
\item $t_1+t_2+t_3 = 2,$
\item $t_1+t_2 \leq 2, t_1+ t_3 \leq 2,t_2+t_3 \leq 2$ and
\item $t_1 \leq 2,t_2 \leq 2 ,t_3 \leq 1.$
\end{itemize}

\section{When $Q$ is a chain}

Our goal in this section is to study the $(P,Q)$-subposet vectors when $Q$ is a chain of $P$. More precisely, we will show that there is a bijection between $(P,Q)$-subposet vectors and integer lattice points of a certain generalized permutohedron constructed from the pair $(P,Q)$. Given a $(P,Q)$-subposet vector $(c_1,\ldots,c_r)$, we are going to look at the vector $(c_1,c_2-c_1,\ldots,c_r-c_{r-1},|P|-c_r)$. We define $M_{P,Q}$ to be the collection of such vectors.

Let us denote the elements of $Q$ as $q_1,\ldots,q_r$ such that $q_1 < \cdots < q_r$ in $P$. 

\begin{remark}
\label{rem:pad}
We are going to add a minimal element $\hat{0}$ and a maximal element $\hat{1}$ to $P$. This does not change the structure of $M_{P,Q}$ or $(P,Q)$-subposet vectors, since all linear extensions would assign same numbers to $\hat{0}$ and $\hat{1}$. We will denote $\hat{0}$ and $\hat{1}$ as $q_0$ and $q_{r+1}$ for technical convenience.
\end{remark}

The following is another way to think of vectors of $M_{P,Q}$. An \newword{order ideal} of $P$ is a subset $I$ of $P$ such that if $x \in I$ and $y \leq x$, then $y \in I$. Now given a linear extension $\sigma : P \rightarrow [|P|]$, we define the order ideals $J_i$ to be the collection of elements $p \in P$ such that $\sigma(p) \leq \sigma(q_i)$ for $0 \leq i \leq r+1$. If we define $I_i$ to be $J_i \setminus J_{i-1}$ for $1 \leq i \leq r+1$, then $(|I_1|,\ldots,|I_{r+1}|)$ is an element of $M_{P,Q}$. Also, any element $(c_1,\ldots,c_{r+1})$ of $M_{P,Q}$ would actually come from some linear extension $\sigma$ and its corresponding sequence of order ideals $J_0 \subset J_1 \subset \cdots \subset J_{r+1}$.

We define the subset $B_{i,i}$ as:
$$B_{1,1} := \{p \in P | q_ 0 \leq p \leq q_1 \},$$
$$B_{i,i} := \{p \in P | q_{i-1} < p \leq q_i \}, i \not = 1.$$ 
For $i < j$, we define the set $B_{i,j}$ as:
$$B_{i,j} := \{p \in P | q_{i-1} < p < q_j, q_i \not < p, p \not < q_{j-1} \}.$$

Then we get a decomposition of $P$ into $B_{i,j}$'s for $1 \leq i \leq j \leq r+1$. Let us define the generalized permutohedron $N_{P,Q}$ as:
$$N_{P,Q} := \sum_{1 \leq i \leq j \leq r+1} |B_{i,j}| \Delta_{[i,j]}.$$

\begin{lemma}
\label{lem:chain1}
Every integer lattice point of $N_{P,Q}$ is a member of $M_{P,Q}$.
\end{lemma} 

\begin{proof}
Let $p=(p_1,\ldots,p_{r+1})$ be an integer lattice point of $N_{P,Q}$. By proposition 14.12 of \cite{Postnikov01012009}, $p$ is the sum of $p_{[i,j]}$'s, where each $p_{[i,j]}$ is an integer lattice point of $|B_{i,j}|\Delta_{[i,j]}$. Each $p_{[i,j]}$ can be expressed as $\sum_{k \in [i,j]} b_{i,j,k} e_k$, where $b_{i,j,k}$'s are nonnegative integers such that $\sum_k b_{i,j,k} = |B_{i,j}|$. We then decompose the set $B_{i,j}$ into $B_{i,j,k}$'s such that:
\begin{enumerate}
    \item for any $c$ and $d$ such that $i \leq c<d \leq j$, all elements of $B_{i,j,c}$ are smaller than any element of $B_{i,j,d}$ in $P$ and,
    \item cardinality of each $B_{i,j,k}$ is given by $b_{i,j,k}$.
\end{enumerate}

Since $p_{[i,j]} = \sum_{k \in [i,j]} b_{i,j,k} e_k$, we have $p = \sum_k \sum_{i,j} b_{i,j,k} e_k$. This tells us that $p_k = \sum_{i,j} b_{i,j,k}$ for all $k$ from $1$ to $r+1$. We define the set $I_k$ to be the union of $B_{i,j,k}$'s for all possible $i$ and $j$'s. If $\{\hat{0}\} \subset I_1 \subset I_1 \cup I_2 \subset \cdots \subset I_1 \cup \cdots \cup I_{r+1} = P$ is a chain of order ideals, then we know that $p=(|I_1|,\ldots,|I_r|,|I_{r+1}|)$ is a member of $M_{P,Q}$, due to the argument just after Remark~\ref{rem:pad}.

So we need to show that there is some way to decompose $P$ into $B_{i,j,k}$'s such that $I_1, I_1 \cup I_2,\ldots, I_1 \cup \cdots \cup I_{r+1}$ are order ideals of $P$. In other words, for any pair $(x,y)$ such that $x \in I_k$ and $y \in I_{k'}$ for $k>k'$, we must have $x \not < y$ in $P$. 

For the sake of contradiction, let us assume we do have elements $x \in B_{i,j,k}$ and $y \in B_{i',j',k'}$ such that $k > k'$ but $x < y$ in $P$. Let us call such pair $(x,y)$ an \newword{inversion}. Looking at all inversion pairs, construct a collection $\C$ by collecting all $(x,y)$'s such that $k-k'$ is minimal. And among the pairs of $\C$, find a pair $(x,y)$ such that there does not lie a $z$ such that $(z,x) \in \C$ or $(y,z) \in \C$. Now let us show that we can switch $x$ and $y$ : to put $x$ in $B_{i,j,k'}$ and $y$ in $B_{i',j',k}$ without introducing any new inversions.

We first need to show that $k,k' \in [i,j] \cap [i',j']$. The fact that $x \in B_{i,j,k}, y \in B_{i',j',k'}$ tells us that:
\begin{itemize}
\item $q_{i-1}< x \leq q_j$,
\item $q_{i'-1} < y \leq q_{j'}$,
\item $k \in [i,j]$ and
\item $k' \in [i',j']$.
\end{itemize}
We also get $q_{i-1} \leq q_{i'-1}$ and $q_j \leq q_{j'}$ from $x <y$ and the definition of $B_{i,j}$ and $B_{i',j'}$. Hence we have $i \leq i'$ and $j \leq j'$. Then $k > k'$ allows us to conclude that $k,k' \in [i',j'] \cap [i,j]$.

Next, we are going to show that this switch does not introduce any new inversions. Assume for the sake of contradiction that we get a new inversion $(z,x)$ (The proof for $(y,z)$ case is also similar and will be omitted). Since $(z,x)$ wasn't an inversion before the switch, we have $z < x$ and $z$ has to be in some $I_{k''}$ where $k \geq k'' > k'$. But since $z<x$ implies $z<y$, the minimality of $k-k'$ tells us that $k'' = k$. This implies $(z,y) \in \C$, which contradicts the condition for our choice of $(x,y)$.

By repeating the  switching process, we can get $I_1,\ldots,I_{r+1}$ such that $I_1,I_1 \cup I_2, \ldots, I_1 \cup \cdots \cup I_{r+1}$ are all order ideals. This switching process does not change the cardinality of any $I_i$ for $1 \leq i \leq r+1$, so $|I_i| = p_i$ for all $1 \leq i \leq r+1$. Hence we get the desired result that $p \in M_{P,Q}$.


\end{proof}

\begin{theorem}
\label{thm:main}
The collection $M_{P,Q}$ is exactly the set of integer lattice points of the generalized permutohedron $N_{P,Q}$.
\end{theorem}

\begin{proof}

From Lemma~\ref{lem:chain1}, all we need to do is show that any element of $M_{P,Q}$ is actually an integer lattice point of $N_{P,Q}$. 

Let $p$ be an element of $M_{P,Q}$. Using Lemma~\ref{lem:facets}, we need to show that we have
$\sum p_t = |P|$ and $\sum_{t \in A} p_t \leq \sum_{A \cap [i,j] \not = \emptyset} |B_{i,j}|$ for each subset $A$. The first equation, $\sum p_t = |P|$ follows from the definition since each point of $M_{P,Q}$ is of the form $(c_1,\ldots,c_r,|P|-\sum c_i)$. So all we are left is to show the second inequality for each subset $A$. We can write $\sum_{t \in I} p_t$ as $\sum_{t \in A} |I_t|$, and since $\cup_{t \in A} I_t \subseteq \cup_{t \in A} \cup_{t \in [i,j]} B_{i,j} \subseteq \bigcup_{A \cap [i,j] \not = \emptyset} B_{i,j}$, we get $\sum_{t \in A} |I_t| \leq \bigcup_{A \cap [i,j] \not = \emptyset} |B_{i,j}|$.
\end{proof}




Since each $(P,Q)$-subposet vector $(c_1,\ldots,c_r)$ corresponds to a point $(c_1,c_2-c_1,\ldots,c_r-c_{r-1},|P|-c_r)$ of $M_{P,Q}$, the above theorem allows us to conclude that:

\begin{corollary}
When $P$ is a poset and $Q$ is a chain in $P$, the collection of $(P,Q)$-subposet vectors are in bijection with integer lattice points of the generalized permutohedron $N_{P,Q}$.
\end{corollary}

Actually, we can say a bit more about $(P,Q)$-subposet vectors. Let us define $\Delta'_{I}$ to be the simplex obtained by sending each point $(x_1,\cdots,x_{r+1})$ of a simplex $\Delta_I$ in $\RR^{r+1}$ to $(x_1,x_1+x_2,\ldots,x_1+\cdots+x_r)$ in $\RR^r$. In other words,
\begin{itemize}
\item if $r+1 \not \in I$, then $\Delta'_I$ is the convex hull of $e_1+ \cdots + e_i$'s for $i \in I$ and,
\item if $r+1 \in I$, then $\Delta'_I$ is the convex hull of the origin and  $e_1+ \cdots + e_i$'s for $i \in I \setminus \{r+1\}$.
\end{itemize}

Then we can describe the set of $(P,Q)$-subposet vectors more precisely:

\begin{corollary}
\label{cor:main}
The $(P,Q)$-subposet vectors are exactly the integer lattice points of the polytope
$$N'_{P,Q} := \sum_{1 \leq i \leq j \leq r+1} |B_{i,j}| \Delta'_{[i,j]}$$.
\end{corollary}

\begin{proof}
All entries of $(c_1,\ldots,c_r)$ are integers if and only if all entries of $(c_1,c_2-c_1,\ldots,c_r-c_{r-1},|P|-c_r)$ are integers. So the claim follows from Theorem~\ref{thm:main}.
\end{proof}

Let us end with an example. Consider a poset $P$ given by Figure~\ref{fig:posetNex}. The chain $Q$ is chosen as the elements labeled $a$ and $b$. We label $\hat{0} = q_0, a = q_1, b = q_2, \hat{1} = q_3$. If we restrict all possible linear extensions of $P$ to $q_1$ and $q_2$, we get integer vectors $(2,4),(2,5),(3,4),(3,5),(4,5)$. These points are exactly the $(P,Q)$-subposet vectors. We have $\hat{0},a \in B_{1,1}, b \in B_{2,2}, \hat{1} \in B_{3,3}, y \in B_{1,2}, z \in B_{1,3}$. So $N'_{P,Q} = 2 \Delta_{[1]} + 1 \Delta_{[2]} + \Delta_{[3]} + \Delta_{[1,2]} + \Delta_{[1,3]}$ and this gives us a pentagon, where all integer lattice points are exactly the elements of $M_{P,Q}$.

\begin{figure}[htbp]
	\begin{center}
	 		\includegraphics[width=1.0\textwidth]{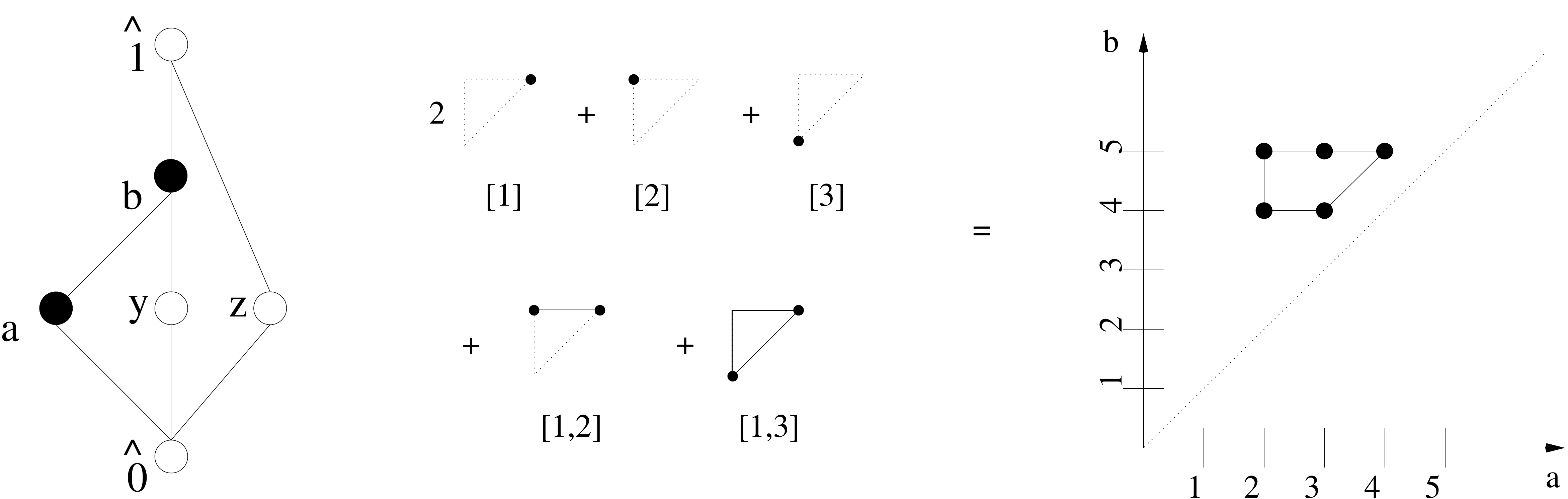}
		\caption{A poset $P$ and chain $Q$ given by the black-colored elements. The summands of $N'_{P,Q}$ and the polytope $N'_{P,Q}$ with its integer lattice points. }
		\label{fig:posetNex}
	\end{center}
\end{figure}

\section{For $Q$ in general}

In this section, we are going to study the $(P,Q)$-subposet vectors when $Q$ is not necessarily a chain of $P$. The $(P,Q)$-subposet vectors are integer lattice points in a union of polytopes combinatorially equivalent to  generalized permutohedron. Then we are going to show that there is a nonconvex, contractible polytope that can be obtained by gluing those polytopes nicely.

\begin{figure}[htbp]
	\begin{center}
	 		\includegraphics[width=0.6\textwidth]{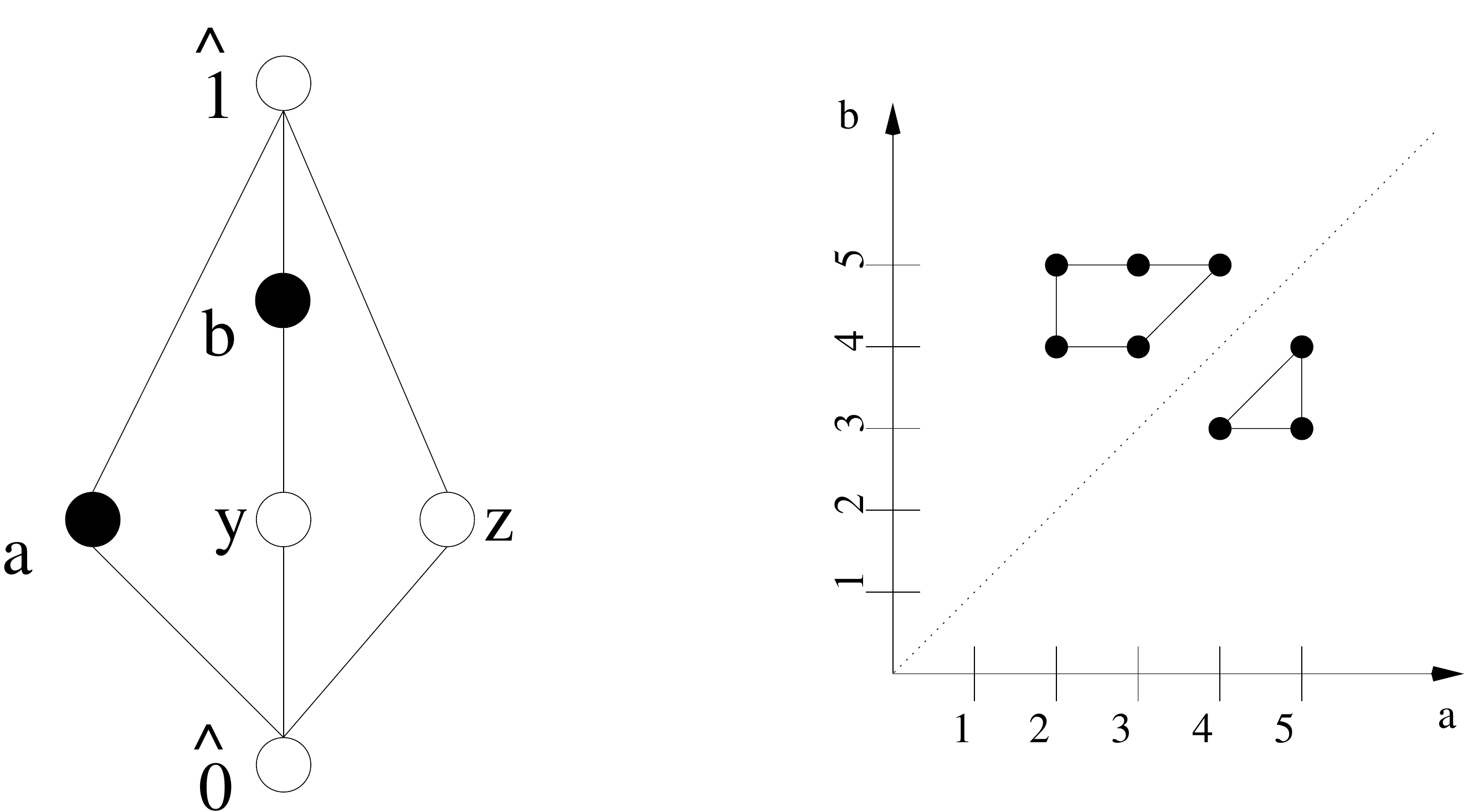}
		\caption{A poset $P$ and subposet $Q$ given by the black-colored elements. $(P,Q)$-subposet vectors are drawn on the right. }
		\label{fig:posetNex3}
	\end{center}
\end{figure}

We will start with an example in Figure~\ref{fig:posetNex3}. One can notice that the points are grouped into two parts, depending on which of $a$ or $b$ is bigger. Let us add the relation $a>b$ to $P$ and $Q$ respectively to get $P_1$ and $Q_1$. Similarly, let us add the relation $a<b$ to $P$ and $Q$ respectively to get $P_2$ and $Q_2$. Then as one can see from Figure~\ref{fig:posetNex2}, one group of points of $(P,Q)$-subposet vectors come from $(P_1,Q_1)$-subposet vectors and the other comes from $(P_2,Q_2)$-subposet vectors.

\begin{figure}[htbp]
	\begin{center}
	 		\includegraphics[width=1.0\textwidth]{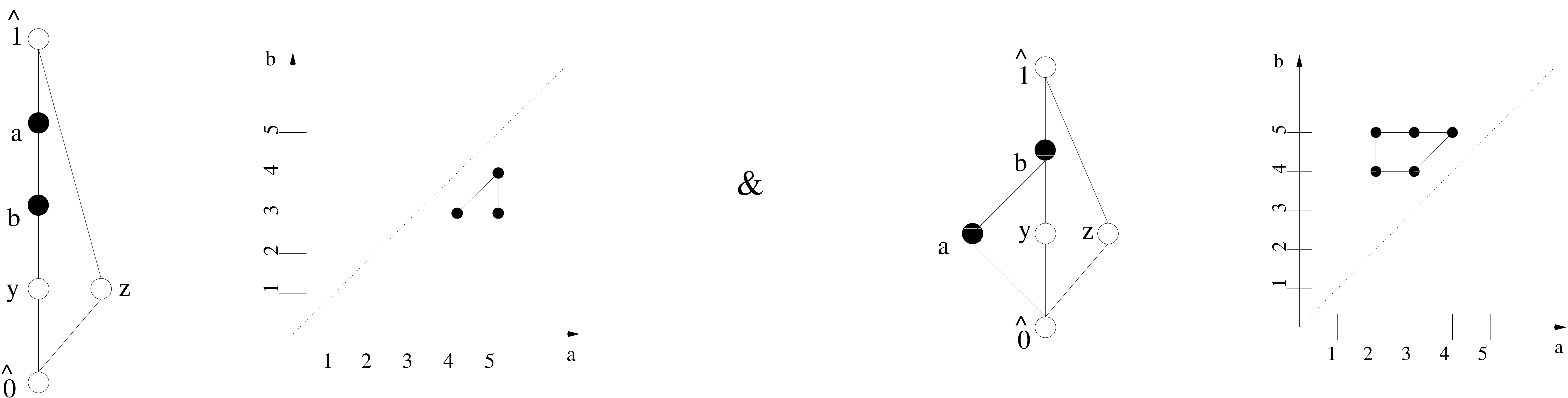}
		\caption{$P_1$ and $Q_1$ are obtained from $P$ and $Q$ by adding $a>b$. $P_2$ and $Q_2$ are obtained from $P$ and $Q$ by adding $a<b$. }
		\label{fig:posetNex2}
	\end{center}
\end{figure}

If we look at the line $x_a=x_b$ in Figure~\ref{fig:posetNex2}, the nearby faces of $N'_{P_1,Q_1}$ and $N'_{P_2,Q_2}$ look identical. More precisely, the intersection of $N'_{P_1,Q_1}$ with $x_a - x_b = 1$ and the intersection of $N'_{P_2,Q_2}$ with $x_a - x_b = -1$ looks identical. And that face looks exactly like $N'_{P_3,Q_3}$ where $P_3$ and $Q_3$ are obtained from $P$ and $Q$ by identifying $a$ and $b$, as in Figure~\ref{fig:posetNex4}.

\begin{figure}[htbp]
	\begin{center}
	 		\includegraphics[width=0.4\textwidth]{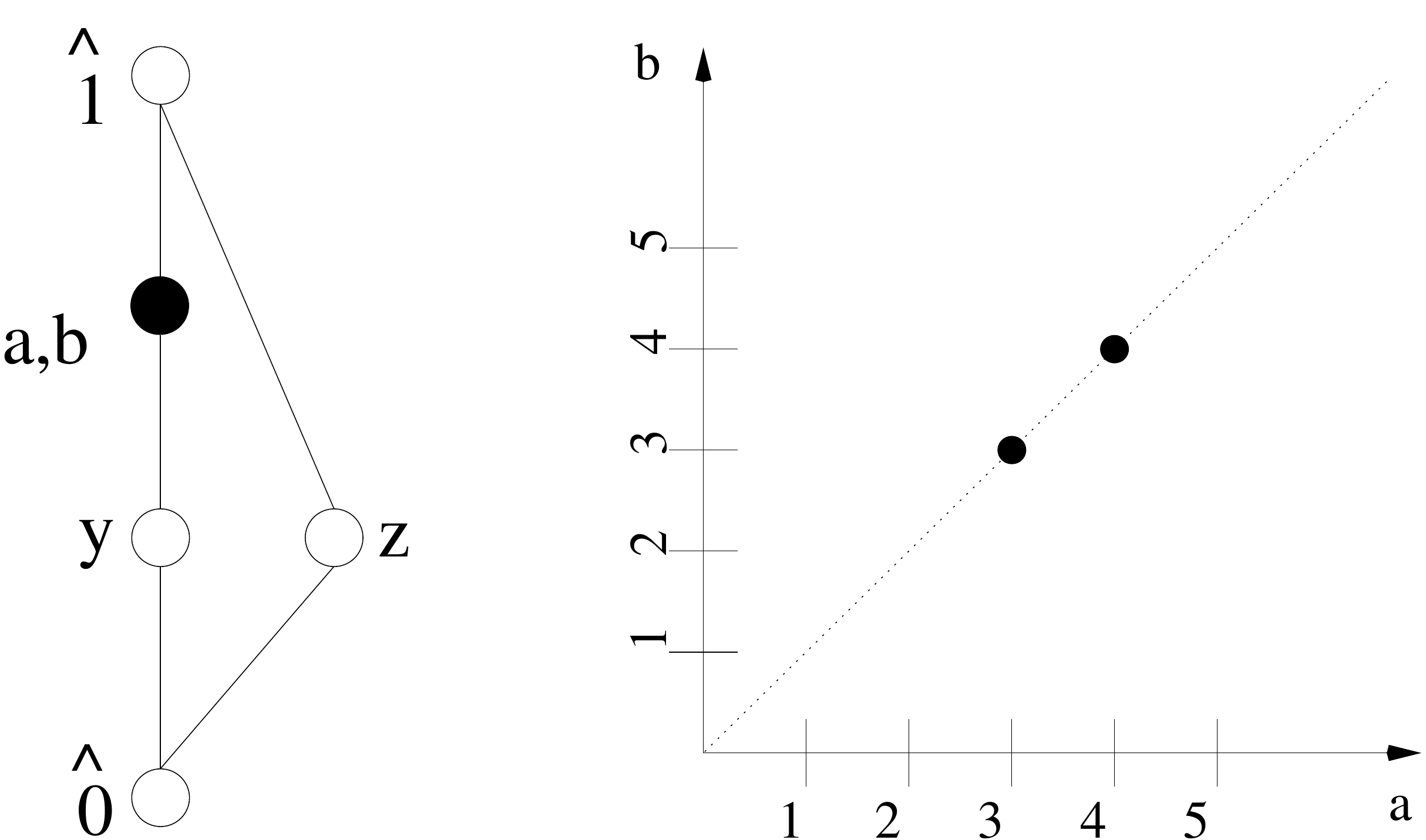}
		\caption{$P_3$ and $Q_3$ are obtained from $P$ and $Q$ by identifying $a$ and $b$. }
		\label{fig:posetNex4}
	\end{center}
\end{figure}

This suggests that we can glue together $N'_{P_1,Q_1}$ and $N'_{P_2,Q_2}$ along $N'_{P_3,Q_3}$. We translate $N'_{P_1,Q_1}$ by negating $1$ from $x_a$ and $N'_{P_2,Q_2}$ by negating $1$ from $x_b$. Then we get a polyhedra as in Figure~\ref{fig:posetNex5}, which we will call the \newword{posetohedron} of the pair $(P,Q)$.

\begin{figure}[htbp]
	\begin{center}
	 		\includegraphics[width=0.4\textwidth]{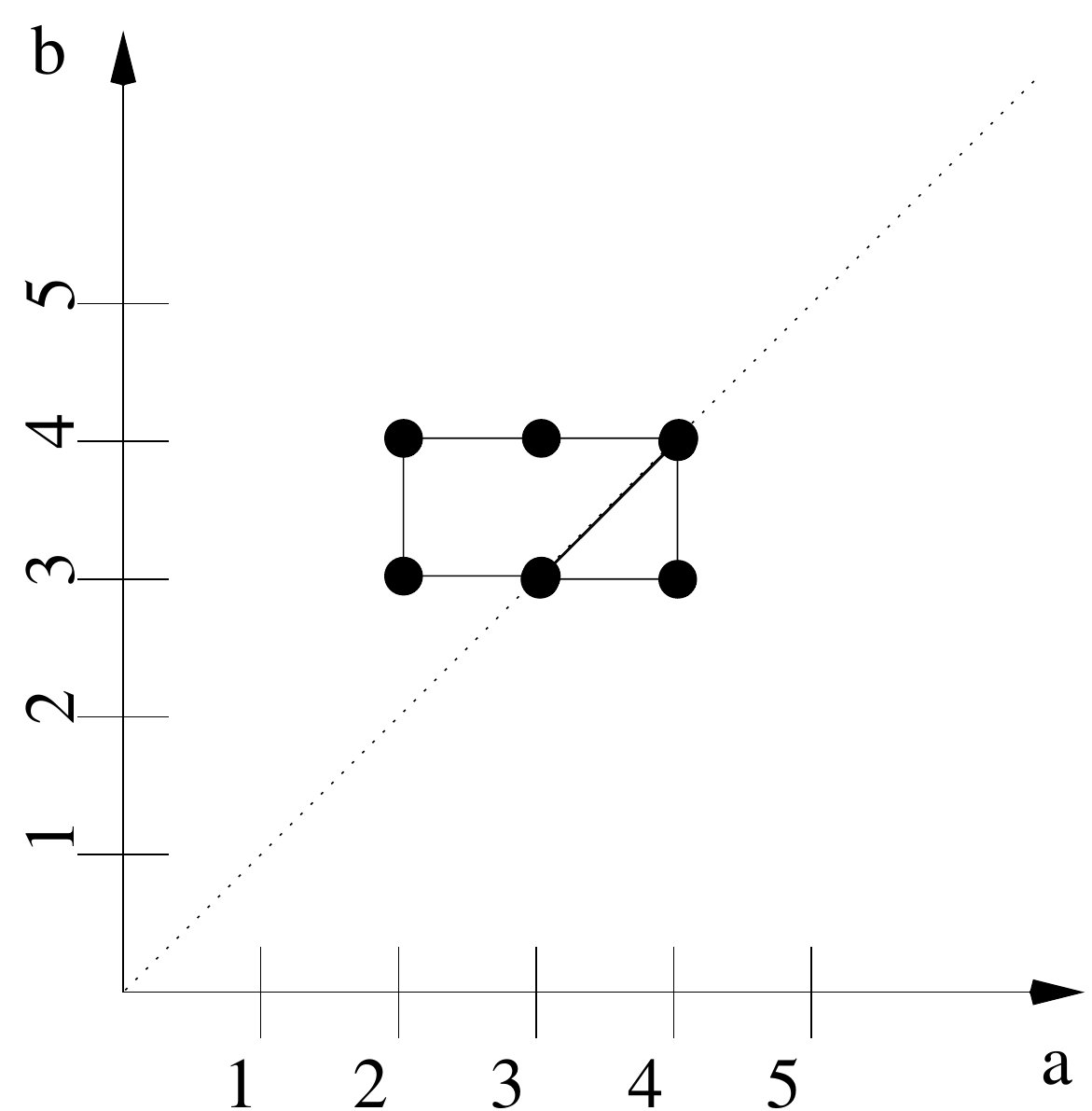}
		\caption{$N'_{P_1,Q_1}$ and $N'_{P_2,Q_2}$ glued together.}
		\label{fig:posetNex5}
	\end{center}
\end{figure}

Now we will describe the above procedure formally. Recall that we denote the elements of $Q$ by $q_1,\ldots,q_r$, by choosing some linear extension on $Q$. We are going to associate a hyperplane arrangement given by hyperplanes $x_i - x_j = 0$ for all pairs $1 \leq i \leq j \leq r$, and denote this by $\A_Q$. Each chamber in $\A_Q$ corresponds to an ordering $x_{w(1)} < \cdots < x_{w(r)}$ where $w \in S_r$. So from now on, we will identify the chambers with their corresponding permutations. We will say that $w$ is \newword{valid} if $q_{w(1)} < \cdots < q_{w(r)}$ is a valid total ordering of elements of $Q$. For $(P,Q)$-subposet vectors coming from linear extensions compatible with the ordering $q_{w(1)} < \cdots < q_{w(r)}$, we denote them $(P,Q,w)$-subposet vectors.

Then the set of $(P,Q)$-subposet vectors is just the disjoint union of $(P,Q,w)$-subposet vectors for all $w \in S_r$, since $(P,Q,w)$-subposet vectors lie in the interior of chamber $w$. If we add the relation $q_{w(1)} < \cdots < q_{w(r)}$ to $P$ and $Q$ to get $P_w$ and $Q_w$ respectively, $(P,Q,w)$-subposet vectors are exactly the integer lattice points of $N'_{P_w,Q_w}$. We will call such polytope a \newword{block}.

We want to show that if we translate each block $N'_{P_w,Q_w}$ by $- \sum_i (i-1)e_{w(i)}$, we get a polytopal complex. In other words, we want to show that under this translation, the blocks glue nicely, especially that the intersection of any collection of blocks is a common face of all such blocks.

We start with the following property:

\begin{lemma}
Let $w$ and $v$ be two different permutations. The translated blocks $N'_{P_w,Q_w} -\sum_i (i-1)e_{w(i)} $ and $N'_{P_v,Q_v} - \sum_i (i-1)e_{v(i)}$ have disjoint interiors.
\end{lemma}

\begin{proof}
The block $N'_{P_w,Q_w}$ lies strictly inside the chamber $x_{w(1)} < \cdots < x_{w(n)}$. Hence $N'_{P_w,Q_w} -i \sum_i e_{w(i)} $ lies in $x_{w(1)} \leq \cdots \leq x_{w(n)}$. Therefore the interiors of $N'_{P_w,Q_w} -i \sum_i e_{w(i)} $ and $N'_{P_v,Q_v} -i \sum_i e_{v(i)}$ cannot overlap.
\end{proof}

We now check that a nonempty intersection of any collection of blocks is a common face of the blocks. We define a \newword{face} of $\A_Q$ to be a face of one of the polyhedron, obtained by taking the closure of a chamber. The faces of $\A_Q$ are in bijection with the faces of a permutohedron under duality. The intersection of some given set of translated blocks happen inside a face of $\A_Q$. Fix a face $F$ of $\A_Q$ and we will use $\C_F$ to denote the set of chambers whose closure contains $F$. A face $F$ which has dimension $d$ corresponds to an ordered partition of $[r]$ into $d$ parts according to \cite{z-lop-95}. To be more precise, $F$ of dimension $d$ corresponds to some ordered partition $\Pi = (\Pi_1,\ldots,\Pi_d)$ and this translates to a partial ordering where $a < b$ if $a \in \Pi_i, b \in \Pi_j$ for $i < j$. The chambers of $\C_F$ are chambers corresponding to the total order compatible with this partial order.

By reordering the coordinates, we may assume that $F$ corresponds to the ordered partition $([1..i_1],[i_1+1..i_1+i_2],\ldots,[i_1+\cdots+i_{d-1}+1..i_1+\cdots+i_d])$. Then each chamber of $\C_F$ correspond to a total ordering $w_1(1) < \cdots < w_1(i_1) < w_2(i_1+1) < \cdots < w_2(i_1+i_2) < \cdots < w_d(i_1+\cdots+i_d)$ where $w_k \in S_{i_k}$ for each $1 \leq k \leq d$.

We first show that the blocks glue nicely when $F$ is a facet of $\A_Q$.

\begin{lemma}
\label{lem:adjblock}
Let $w$ and $v$ be two valid chambers in $\A_Q$ where $w$ is a permutation $[\cdots,i,j,\cdots]$ and $v$ is $[\cdots,j,i,\cdots]$. We denote $H$ to be the hyperplane $x_i - x_j = 0$. Then $(N'_{P_w,Q_w} - i \sum e_{w(i)}) \cap H = (N'_{P_v,Q_v} - i \sum e_{v(i)}) \cap H$.
\end{lemma}

\begin{proof}

Since $q_i$ and $q_j$ are incomparable in $P$, for any integer lattice point of $N'_{P_w,Q_w}$ such that the value of $j$-th coordinate is exactly one larger than the value of $i$-th coordinate, we can swap those entries and get a integer lattice point of $N'_{P_v,Q_v}$, and vice versa. So $N'_{P_w,Q_w} \cap H = N'_{P_v,Q_v}$.

\end{proof}

Our goal now is to use this result to show that all blocks glue nicely.

\begin{lemma}
\label{lem:halfspace}
Let $F$ be a face of $\A_Q$. If $w$ and $v$ are two valid chambers in $\C_F$, then one can go from $w$ to $v$ by sequence of adjacent transpositions, while staying inside $\C_F$.
\end{lemma}

\begin{proof}
Denote the ordered partition corresponding to $F$ to be $([1..i_1],[i_1+1..i_1+i_2],\ldots,[i_1+\cdots+i_{d-1}+1..i_1+\cdots+i_d])$. Then $F$ is given by intersecting the halfspaces $x_a < x_b$ for all pairs $a \in [\sum_1^c i_m+1 ,\sum_1^{c+1} i_m], b \in [\sum_1^d i_m+1 ,\sum_1^{d+1} i_m]$ where $c<d$ and hyperplanes $x_a = x_b$ for all possible pairs $a,b \in [\sum_1^c i_m+1, \sum_1^{c+1} i_m]$. Then chambers of $\C_F$ are exactly the chambers of $\A_Q$ inside the intersection of halfspaces $x_a < x_b$ for all pairs $a \in [\sum_1^c i_m+1 ,\sum_1^{c+1} i_m], b \in [\sum_1^d i_m+1 ,\sum_1^{d+1} i_m]$ where $c<d$. This space corresponds to a chamber of some subarrangement of $\A_Q$. Hence we can go from one chamber of $\C_F$ to another chamber of $\C_F$ by moving across the facets, and by only using the chambers of $\C_F$.
\end{proof}

Now we are ready to prove the following result:

\begin{proposition}
\label{prop:complex}
The collection of translated blocks $N'_{P_w,Q_w} - \sum_i (i-1)e_{w(i)}$'s form a polytopal complex.
\end{proposition}

\begin{proof}
Let $F$ be a face of $\A_Q$. By combining Lemma~\ref{lem:adjblock} and Lemma~\ref{lem:halfspace}, we have that: if $w$ and $v$ are two valid chambers in $\C_F$, then $(N'_{P_w,Q_w} - i \sum_i e_{w(i)}) \cap F = (N'_{P_v,Q_v} - i \sum_i e_{v(i)}) \cap F$. 

Now let us show that for any valid $w$ in some $\C_F$, the intersection $(N'_{P_w,Q_w} - i \sum_i e_{w(i)}) \cap F$ is a face of $N'_{P_w,Q_w} - i \sum_i e_{w(i)}$. As we move from chamber $w$ to $F$, the extra condition we are imposing are bunch of hyperplanes of form $x_a - x_b = 0$. But since these are translates of some hyperplanes that define the polytope $N'_{P_w,Q_w}$, the intersection $(N'_{P_w,Q_w} - i \sum_i e_{w(i)}) \cap F$ can be thought of as changing some inequalities defining $N'_{P_w,Q_w} - i \sum_i e_{w(i)}$ to equalities. Hence $(N'_{P_w,Q_w} - i \sum_i e_{w(i)}) \cap F$ is a face of $N'_{P_w,Q_w} - i \sum_i e_{w(i)}$. This implies that intersection of any set of blocks is a common face of the blocks.
\end{proof}

We will call the support of the polytopal complex the $(P,Q)$-posetohedron. When $Q$ is not a chain, then $(P,Q)$-posetohedron is not convex in general, as one can see from Figure~\ref{fig:posetNex6}.

\begin{figure}[htbp]
	\begin{center}
	 		\includegraphics[width=0.9\textwidth]{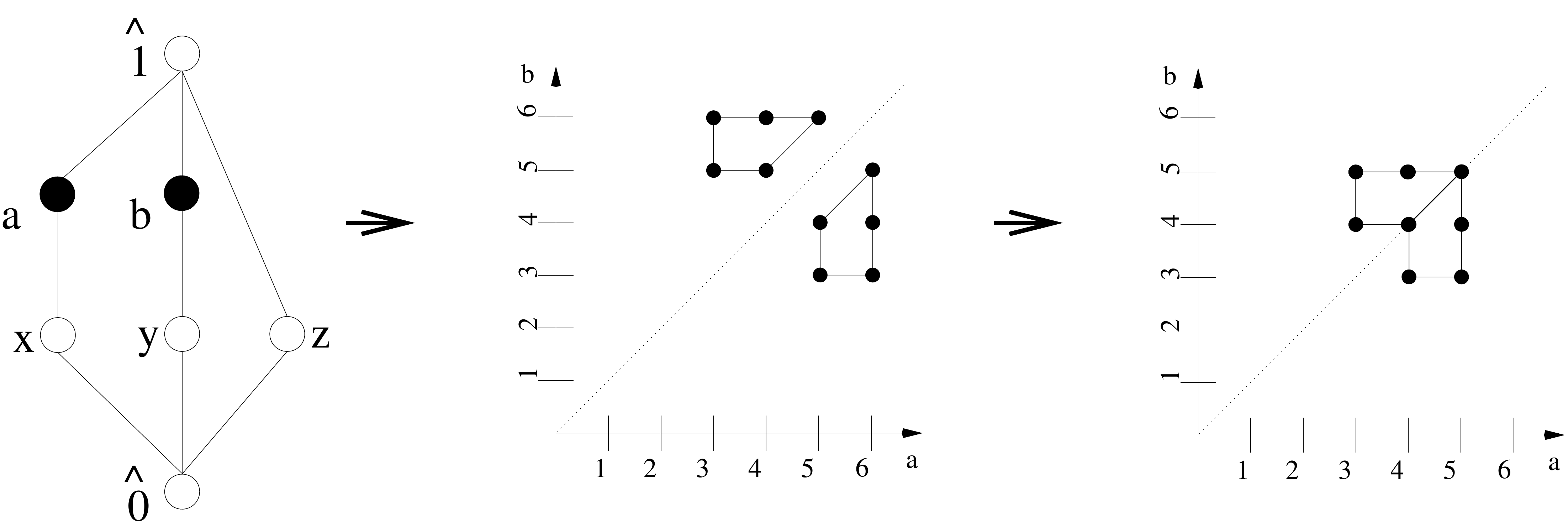}
		\caption{Example of a posetohedron that is not convex.}
		\label{fig:posetNex6}
	\end{center}
\end{figure}

\begin{proposition}
\label{prop:contractible}
The $(P,Q)$-posetohedron is contractible.
\end{proposition}

\begin{proof}

We prove this by induction on the number of independent pairs that $Q$ contains. When $Q$ does not contain any independent pair, that is a chain, then the $(P,Q)$-posetohedron is just a generalized associahedron and is contractible. When $Q$ does contain some independent pair $(i,j)$, we get three pairs of posets by:
\begin{enumerate}
\item the pair $(P_1,Q_1)$ by adding the relation $i<j$ to $P$ and $Q$ respectively,
\item the pair $(P_2,Q_2)$ by adding the relation $i>j$ to $P$ and $Q$ respectively,
\item the pair $(P_3,Q_3)$ by identifying $i$ and $j$ to be a single element $k$, and replacing $i$ and $j$ with $k$ in all relations of $P$ and $Q$.
\end{enumerate}

By induction hypothesis, $(P_1,Q_1)$-posetohedron, $(P_2,Q_2)$-posetohedron and $(P_3,Q_3)$-posetohedron are contractible. The $(P,Q)$-posetohedron is obtained by gluing $(P_1,Q_1)$-posetohedron and $(P_2,Q_2)$-posetohedron where their intersection is combinatorially equivalent to the $(P_3,Q_3)$-posetohedron. Since they are all contractible, $(P,Q)$-posetohedorn is also contractible.
\end{proof}

Hence we get a non-convex polytope from $(P,Q)$-subposet vectors, and this polytope turns out to be contractible.

\begin{problem}
Is there some interesting topological property of a $(P,Q)$-posetohedron that depends on the combinatorics of $P$ and $Q$?
\end{problem}

\section{Describing the vertices of a posetohedron}

In this section we use the machinery from \cite{Postnikov01012009}
to give a description for the vertices of $N_{P,Q}$. Since the general case is obtained by gluing the posetohedra when $Q$ is a chain, we will restrict ourselves to when $Q = \{ \hat{0}=q_{0}<q_{1}<...<q_{r+1}=\hat{1} \}$ is a chain.

Recall that a generalized permutohedron can be expressed by $\sum_{1 \leq i,j \leq r+1} c_{i,j}\Delta_{[i,j]}$, where $c_{i,j}$ are nonnegative integers. In case when $c_{i,j}>0$ for all $i$ and $j$, the vertices of of the polytope are in bijection with plane binary trees on $[r+1]$ with the \newword{binary search labeling} \cite{Knuth:1998:ACP:280635}. Binary search labeling is the unique labeling of the tree nodes such that the label of any node is greater than that of any left descendant, and
less than that of any right descendant. Let $T$ be such a binary tree, and identify any of its nodes with its labeling. Extending Corollary 8.2 of \cite{Postnikov01012009}, we get:
\begin{lemma}
\label{lem:vertices} The vertex $v_{T}=(t_{1},...,t_{r+1})$ of a generalized permutohedron $\sum_{1 \leq i,j \leq r+1} c_{i,j}\Delta_{[i,j]}$ is given by

$$t_{k}=\sum_{l_{k}\leq i\leq k\leq j\leq r_{k}}c_{i,j},$$
where $l_{k},r_{k}$ are such that the interval $[l_{k},r_{k}]$ is
exactly the set of descendants of $k$ in $T$.
\end{lemma}

\begin{figure}[htbp]
	\begin{center}
	 		\includegraphics[width=0.9\textwidth]{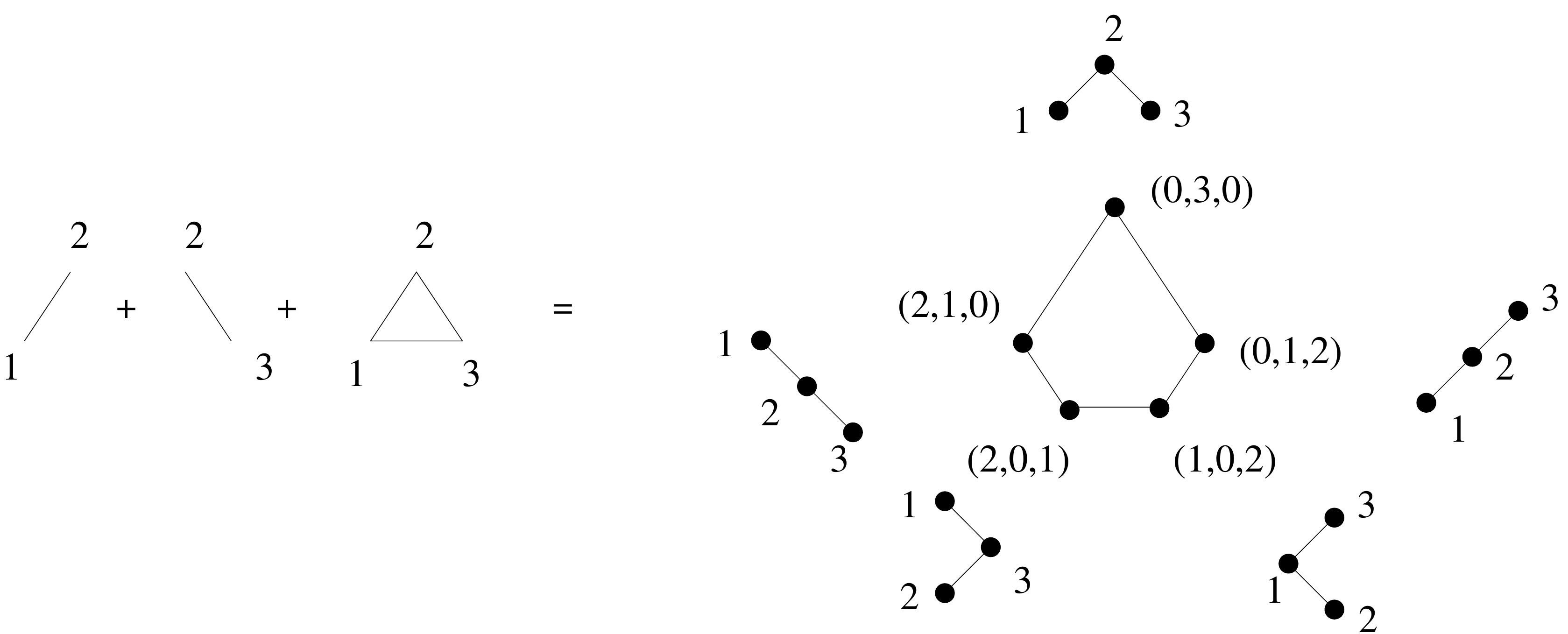}
		\caption{Using binary search labeling to describe the vertices of $\Ass_3$.}
		\label{fig:ass3ver}
	\end{center}
\end{figure}

There is a well-known bijection between plane binary trees on $[r+1]$
and subdivisions of the shifted triangular $(r+1)$-by-$(r+1)$
shape $D_{r+1}$ into rectangles, each touching a diagonal box. The
nice feature about this bijection is that if we denote by $R_{k}$
the rectangle containing the $k$th diagonal box, then 

$$l_{k}\leq i\leq k\leq j\leq r_{k}\ \ \Longleftrightarrow(i,j)\in R_{k}.$$

Figure~\ref{fig:bijection} shows an example of this bijection.

\begin{figure}[htbp]
	\begin{center}
	 		\includegraphics[width=0.4\textwidth]{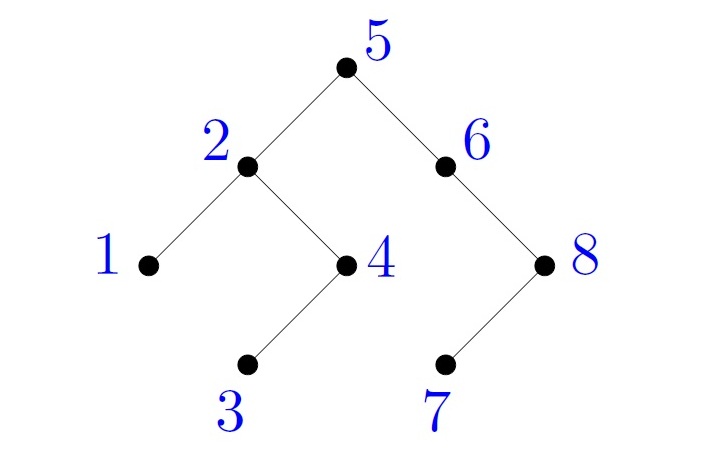}
	 		\includegraphics[width=0.4\textwidth]{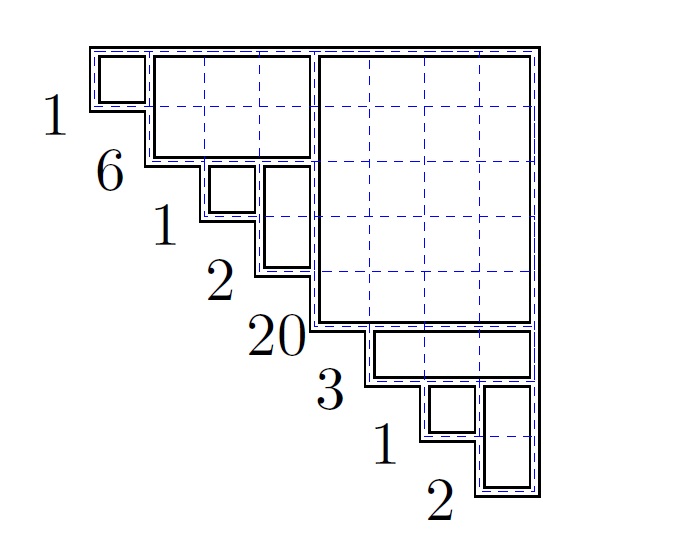}
		\caption{A binary search labeling and a corresponding subdivision.}
		\label{fig:bijection}
	\end{center}
\end{figure}


Writing the numbers $|B_{i,j}|$ in the boxes of the triangular shape $D_{r+1}$, we obtain a nice way to visualize the result of \ref{lem:vertices}.

%

%

\begin{corollary}
Consider a subdivision $\Xi$ of $D_{r+1}$ into rectangles $R_{1},...,R_{r+1}$
with $(i,i)\in R_{i}$. Then a vertex $v_{\Xi}=(t_{1},...,t_{r+1})$
of $N_{P,Q}$ is given by

\[
t_{k}=\sum_{(i,j)\in R_{k}}|B_{i,j}|\]

The map $\Xi\mapsto v_{\Xi}$ is always surjective, and it is a bijection
if and only if $|B_{i,j}|>0$ for all $i<j$. 
\end{corollary}

This corollary also suggests a nice way of constructing linear extensions
of $P$, whose $(P,Q)$-subposet vector is the vertex $v_{\Xi}$ of the
posetohedron: Fill rectangle $R_{k}$ with the numbers $t_{1}+...+t_{k-1}+1,...,t_{1}+...+t_{k}$
(i.e. construct an order preserving bijection $\sigma_{k}:\cup_{(i,j)\in R_{k}}B_{i,j}\rightarrow[t_{1}+...+t_{k-1}+1,t_{1}+...+t_{k}]$
for each k, and then combine the $\sigma_{k}$'s to produce a linear extension
of $P$).

Note that for each of the $C_{r+1}=\frac{1}{r+2}{2(r+1) \choose r+1}$
subdivisions $\Xi$ of $D_{r+1}$ produces a vertex of $N_{P,Q}$,
but some of these will coincide if some of the $|B_{i,j}|$ are 0.
For example, if $r=2$ and $b_{13}=0$, then $N_{P,Q}$ will have
4 vertices:

\begin{eqnarray*}
v_{1} & = & \left(b_{11}+b_{12},b_{22}+b_{23},b_{33}\right),\\
v_{2} & = & \left(b_{11}+b_{12},b_{22},b_{23}+b_{33}\right),\\
v_{3} & = & \left(b_{11},b_{12}+b_{22}+b_{23},b_{33}\right),\\
v_{4} & = & \left(b_{11},b_{12}+b_{22},b_{23}+b_{33}\right).\end{eqnarray*}

\bibliographystyle{plain}
\bibliography{posetpoly} 

\end{document}